\documentclass[12pt]{article}
\usepackage{stmaryrd}
\usepackage{amssymb}
\usepackage{mathrsfs}
\usepackage[hypertex]{hyperref}
\usepackage{amsfonts}
\usepackage{graphicx}
\usepackage{mathptmx}
\usepackage{latexsym,amsmath,amssymb,amsfonts,amsthm}

\newtheorem{theorem}{Theorem}[section]
\newtheorem{lemma}[theorem]{Lemma}

\newtheorem{corollary}[theorem]{Corollary}

\newtheorem{remark}[theorem]{Remark}

\newtheorem{defn}[theorem]{Definition}

\begin{document}
\setcounter{page}{1}
\title{A class of rotationally symmetric quantum layers of dimension 4}
\author{Jing Mao}
\date{}
\protect\footnotetext{\!\!\!\!\!\!\!\!\!\!\!\!{ MSC 2010: 34L05;
81Q10}
\\
{ ~~Key Words: Discrete spectrum; Essential spectrum; Bound state; Quantum layer} \\
  Supported by Funda\c{c}\~{a}o
para a Ci\^{e}ncia e Tecnologia (FCT) through a doctoral fellowship
SFRH/BD/60313/2009.}
\maketitle ~~~\\[-15mm]
\begin{center}{\footnotesize  Departamento de Matem\'{a}tica,
Instituto Superior T\'{e}cnico, Technical University of Lisbon,
Edif\'{\i}cio Ci\^{e}ncia, Piso 3, Av.\ Rovisco Pais, 1049-001
Lisboa, Portugal; jiner120@163.com, jiner120@tom.com}
\end{center}

%\\[5mm]
\begin{abstract}
Under several geometric conditions imposed below, the existence of
the discrete spectrum below the essential spectrum is shown for the
Dirichlet Laplacian on the quantum layer built over a spherically
symmetric hypersurface with a pole embedded in the Euclidean space
$R^{4}$. At the end of this paper, we also show the advantage and
independence of our main result comparing with those existent
results for higher dimensional quantum layers or quantum tubes.
\end{abstract}

\markright{\sl\hfill  J. Mao \hfill}

\section{Introduction}
\renewcommand{\thesection}{\arabic{section}}
\renewcommand{\theequation}{\thesection.\arabic{equation}}
\setcounter{equation}{0} \setcounter{maintheorem}{0}

The study of the spectral properties of the Dirichlet Laplacian in
infinitely stretched regions has attracted so much attention, since
it has applications in elasticity, acoustics, electromagnetism, etc.
It also has application in the quantum physics. Since Duclos et al.
considered the existence of the discrete spectrum of the Dirichlet
Laplacian of the quantum layers built over surfaces in \cite{ppd},
many similar results have been obtained for the quantum layers whose
reference manifolds are surfaces. However, very little was known
about the existence of the discrete spectrum of the Dirichlet
Laplacian on the quantum layers of dimension greater than 3.

In \cite{ll,ll2}, under some geometric assumptions therein, Lin and
Lu have successfully proved the the existence of the discrete
spectrum of the Dirichlet Laplacian on the quantum layers built over
 submanifolds of the Euclidean space
$R^{m}$ ($3\leq{m}<\infty$). However, the parabolicity of the
reference submanifold and nonpositivity of the integration of
$\mathcal{K}_{m-2}$ defined by (\ref{4.1}) are necessary. Is the
parabolicity of the reference submanifold necessary for the
existence of the discrete spectrum? We try to give a negative answer
here. In general, it is not easy to judge whether a prescribed
manifold is parabolic or not. However, Grigor'yan has shown a
sufficient and necessary condition, which is related to the area of
the boundary of the geodesic ball and could be easily computed, of
parabolicity for spherically symmetric manifolds in \cite{grig1}.
Hence, we guess maybe we can expect to get the existence of the the
discrete spectrum of the quantum layer built over some  spherically
symmetric submanifold, which is non-parabolic, of the Euclidean
space $R^{m}$ ($3\leq{m}<\infty$).

In order to state our main result, we define two quantities
$\sigma_{0}$ and $\sigma_{ess}$ as follows.
\begin{defn} Let $M$ be a manifold whose Laplacian $\Delta$
can be extended to a self-adjoint operator. Let
\begin{eqnarray} \label{1.1}
\sigma_{0}=\inf_{f\in{C}^{\infty}_{0}(M)}\frac{-\int_{M}f\Delta{f}dV_{M}}{\int_{M}f^{2}dV_{M}},
\end{eqnarray}
\begin{eqnarray}  \label{1.2}
\sigma_{ess}=\sup_{K}\inf_{f\in{C}^{\infty}_{0}(M\setminus{K})}\frac{-\int_{M}f\Delta{f}dV_{M}}{\int_{M}f^{2}dV_{M}},
\end{eqnarray}
where $K$ is running over all compact subsets of $M$, and $dV_{M}$
denotes the volume element of $M$.
\end{defn}

In fact, $\sigma_{0}$ and $\sigma_{ess}$ are the lower bound of the
spectrum and the lower bound of the essential spectrum of the
Laplacian $\triangle$ on $M$, respectively. In general case,
$\sigma_{0}\leq\sigma_{ess}$. If $\sigma_{0}<\sigma_{ess}$, then the
existence of the discrete spectrum is obvious. In mathematical
physics, points in the discrete spectrum are called bound states,
moreover, the lowest bound state is called the ground state.

We want to show that $\sigma_{0}<\sigma_{ess}$ holds for the the
Laplacian $\triangle$ of the class of quantum layers considered in
the sequel. In fact, by using this strategy we can prove the
following.

\begin{theorem} (Main theorem)\label{theorem1}
Assume $\Sigma$ is a spherically symmetric hypersurface with a pole
embedded in $R^{4}$, and $\Sigma$ is not a hyperplane, if in
addition $\mathcal{K}_{2}$ is integrable on $\Sigma$ and
\begin{eqnarray*}
\int_{\Sigma}\mathcal{K}_{2}d\Sigma\leq0
\end{eqnarray*}
 with $\mathcal{K}_{2}$ defined by
(\ref{4.1}), then under assumptions A1, A2 and A3 given in Section
2, the ground state of the quantum layer $\Omega$ built over
$\Sigma$ exists.
\end{theorem}

The paper is organized as follows. The fundamental geometric
properties of the quantum layers built over spherically symmetric
hypersurfaces will be discussed in the next section. The fact that
the Laplacian $\triangle$ on the quantum layers can be extended to a
self-adjoint operator will be explained in Section 3. The main
theorem above will be proved in Section 4.

\section{Geometry of rotationally symmetric quantum layers}
\renewcommand{\thesection}{\arabic{section}}
\renewcommand{\theequation}{\thesection.\arabic{equation}}
\setcounter{equation}{0} \setcounter{maintheorem}{0}

Let $m$ ($3\leq{m}<\infty$) be an integer and let $\Sigma$ be a
$C^{2}$-smooth hypersurface with a pole embedded in the Euclidean
space $R^{m}$. The existence of a pole on $\Sigma$ is a nontrivial
assumption under which $\Sigma$ is necessarily diffeomorphic to
$R^{m-1}$ leading to the simple connectedness and non-compactness of
$\Sigma$. Under this assumption, we can also set up the global
geodesic polar coordinates to parametrize the hypersurface $\Sigma$
by a unique patch $p:\Sigma_{0}\rightarrow{R^{m}}$, where
$\Sigma_{0}:=(0,\infty)\times{\mathbb{S}^{m-2}}$ with
$\mathbb{S}^{m-2}$ the unit sphere in $R^{m-1}$. Naturally, $\Sigma$
can be identified with the image of $\Sigma_{0}$. The tangent
vectors $p_{,\mu}:=\frac{\partial{p}}{\partial{q}^{\mu}}$ are
linearly independent and span the tangent space at every point of
$\Sigma$, correspondingly, the unit normal vector field $\vec{n}$
can be determined. Let $\Omega_{0}:=\Sigma_{0}\times(-a,a)$, then
the quantum layer $\Omega:=\Phi(\Omega_{0})$ of width $2a$ built
over $\Sigma$ can be defined by a natural mapping
$\Phi:\Omega_{0}\rightarrow{R^{m}}$ as follows
\begin{eqnarray} \label{2.1}
\Phi(q,u):=p(q)+u\vec{n}(q), \qquad (q,u)\in\Sigma_{0}\times(-a,a).
\end{eqnarray}

We make an agreement on the indices range,
$1\leq{\mu,\nu,\ldots,}\leq{m-1}$ and $1\leq{i,j,\ldots,}\leq{m}$.
Denote the pole on $\Sigma$ by $o$, we know that the exponential map
$\exp_{o}:\mathcal{D}_{o}\rightarrow{\Sigma}$ is a diffeomorphism,
where
$\mathcal{D}_{o}=\{s\xi|~0\leq{s}<\infty,~\xi\in{S^{m-2}_{o}}\}$
with $S^{m-2}_{o}$ the unit sphere in the tangent space
$T_{o}(\Sigma)$. For a fixed vector $\xi\in{T_{o}M}$, $|\xi|=1$, let
$\xi^{\bot}$ be the orthogonal complement of $\{\mathbb{R}\xi\}$ in
$T_{o}M$ and let $\tau_{s}:T_{o}M\rightarrow{T_{exp_{o}(s\xi)}M}$ be
the parallel translation along the geodesic
$\gamma_{\xi}(s):=exp_{o}(s\xi)$ with $\gamma'(0)=\xi$. Define
 the path of linear transformations $\mathbb{A}(s,\xi):\xi^{\bot}\rightarrow{\xi^{\bot}}$ by
 \begin{eqnarray*}
\mathbb{A}(s,\xi)\eta=(\tau_{s})^{-1}Y(s),
 \end{eqnarray*}
where $Y(s)$ is the Jacobi field along $\gamma_{\xi}$ satisfying
$Y(0)=0$, $(\triangledown_{s}Y)(0)=\eta$. Moreover, for
$\eta\in{\xi^{\bot}}$, set
\begin{eqnarray*}
\mathcal{R}(s)\eta=(\tau_{s})^{-1}\mathbb{R}(s)(\tau_{s}\eta)=(\tau_{s})^{-1}R(\gamma'_{\xi}(s),\tau_{s}\eta)
\gamma'_{\xi}(s),
\end{eqnarray*}
then $\mathcal{R}(s)$ is a self-adjoint map of $\xi^{\bot}$, where
the curvature tensor is given by $R(X,Y)Z=-[\nabla_{X},
\nabla_{Y}]Z+ \nabla_{[X,Y]}Z$. Obviously, the map
$\mathbb{A}(s,\xi)$ satisfies the Jacobi equation
 $\mathbb{A}''+\mathcal{R}\mathbb{A}=0$ with initial conditions
 $\mathbb{A}(0,\xi)=0$,$\mathbb{A}'(0,\xi)=I$, and by applying Gauss's lemma the Riemannian metric of $M$ can be
 expressed by
 \begin{eqnarray*}
 dt^{2}(exp_{o}(s\xi))=ds^{2}+|\mathbb{A}(s,\xi)d\xi|^{2}
 \end{eqnarray*}
 on the set
 $exp_{o}(\mathcal{D}_0)$. Hence, the induced
metric $g_{\mu\nu}$ in the geodesic polar coordinates satisfies
\begin{eqnarray*}
\sqrt{\det[g_{\mu\nu}]}=\det\mathbb{A}(s,\xi).
\end{eqnarray*}
 Define a function $J>0$ on
 $\mathcal{D}_{0}$ by
 \begin{eqnarray*}
 J^{m-2}=\sqrt{\det[g_{\mu\nu}]},
 \end{eqnarray*}
that is $dV_{\Sigma}=J^{m-2}dsd\xi$. We know that the function
$J(s,\xi)$ satisfies (cf. \cite{p}, p. 244)
\begin{eqnarray*}
&& J''+\frac{1}{(m-2)}Ricci\left(\frac{d}{d s},
\frac{d}{d s}\right)J\leq 0  \\
&&J(s,\xi)=s + O(s^2) \\
&& J'(s,\xi)=1+O(s),
\end{eqnarray*}
where $Ricci$ denotes the Ricci curvature tensor on $\Sigma$ and
$\frac{d}{d s}$ is the radial unit tangent vector along the geodesic
$\gamma_{\xi}(s)$. So, we have
\begin{eqnarray} \label{2.2}
J''+\frac{1}{(m-2)}Ricci\left(\frac{d}{d s}, \frac{d}{d
s}\right)J\leq 0 \qquad {\rm{with}} \quad J(0,\xi)=0, \quad
J'(0,\xi)=1.
\end{eqnarray}

 Consider now layers which are invariant with
respect to rotations around a fixed axis in $R^{m}$. We could thus
suppose that $\Sigma$ is a rotational hypersurface parametrized by
$p:\Sigma_{0}\rightarrow{R^{m}}$,
\begin{eqnarray} \label{2.3}
p(s,\theta_{1},\ldots,\theta_{m-2}):=(r(s)\cos(\theta_{1}),r(s)\sin(\theta_{1})\cos(\theta_{2}),r(s)\sin(\theta_{1})\sin(\theta_{2})
\cos(\theta_{3}),\ldots,\nonumber\\
r(s)\sin(\theta_{1})\ldots{\sin}(\theta_{m-3})\cos(\theta_{m-2}),r(s)\sin(\theta_{1})\ldots{\sin}(\theta_{m-3})\sin(\theta_{m-2}),z(s)),
\end{eqnarray}
where $r,z\in{C^{2}}((0,\infty))$, $r>0$ and
$(\theta_{1},\ldots,\theta_{m-2})\in{\mathbb{S}^{m-2}}$. This
parametrization will be the geodesic polar coordinate chart if we
additionally require
\begin{eqnarray} \label{2.4}
\left(r'(s)\right)^{2}+\left(z'(s)\right)^{2}=1,
\end{eqnarray}
since by direct calculation the induced metric tensor on $\Sigma$
can be written as $ds^{2}+r^{2}|d\xi|^{2}$ with
\begin{eqnarray*}
|d\xi|^{2}:=d\theta_{1}^{2}+(\sin\theta_{1})^{2}d\theta_{2}^{2}+(\sin\theta_{1})^{2}(\sin\theta_{2})^{2}d\theta_{3}^{2}+\cdots+
(\sin\theta_{1})^{2}(\sin\theta_{2})^{2}\cdots(\sin\theta_{m-3})^{2}d\theta_{m-2}^{2},
\end{eqnarray*}
the round metric on $\mathbb{S}^{m-2}$, provided the requirement
(\ref{2.4}) is satisfied. So, we have $dV_{\Sigma}=r^{m-2}dsd\xi$,
which implies the function $J$ defined above satisfies $J=r$ in this
case. Moreover, under the parametrization (\ref{2.3}) with the
requirement (\ref{2.4}), $\Sigma$ is a spherically symmetric
hypersurface with a pole, and its Weingarten tensor is given by
$(h_{\mu\nu})={\rm{diag}}(k_{s},k_{\theta_{1}}\ldots,k_{\theta_{m-2}})$
with the principle curvatures
 \begin{eqnarray} \label{2.5}
 k_{s}=r'z''-r''z' \quad{\rm{and}}\quad
k_{\theta}:=k_{\theta_{1}}=\cdots=k_{\theta_{m-2}}=\frac{z'}{r}.
\end{eqnarray}
As pointed out in \cite{ppd}, it is sufficient to know the function
$s\rightarrow{k_{s}(s)}$ only, since $r,z$ can be constructed from
the relations
\begin{eqnarray*}
r(s)=\int_{0}^{s}\cos b(\vartheta)d\vartheta, \qquad
z(s)=\int_{0}^{s}\sin b(\vartheta)d\vartheta,
\end{eqnarray*}
with $b(\vartheta):=\int_{0}^{s}k_{s}(\vartheta)d\vartheta$.

By (\ref{2.2}), (\ref{2.4}), (\ref{2.5}) and the facts $J=r$ and
$Ricci\left(\frac{d}{ds},\frac{d}{ds}\right)=(m-2)k_{s}k_{\theta}$,
we know that the function $r(s)$ satisfies
\begin{eqnarray} \label{2.6}
r''+k_{s}k_{\theta}r= 0 \qquad {\rm{with}} \quad r(0)=0, \quad
r'(0)=1.
\end{eqnarray}
This equation will make an important role in the proof of Theorem
\ref{theorem1}.

In the sequel, we impose the following assumptions on $\Sigma$.
 \vspace{2mm}
 $\\$ A1. $\Sigma$ is not self-intersecting, i.e., $\Phi$ is
injective.
 $\\$ A2. The half width  $a$ of the layer satisfies
$a<\rho_{m}:=\left(\max\{\|k_{s}\|_{\infty},\|k_{\theta}\|_{\infty}\}\right)^{-1}$,
where $\|\cdot\|_{\infty}$ denotes the $L^{\infty}$-norm.
 $\\$ A3. For $x\in\Sigma$, $\|A\|(x)\rightarrow0$ as
$d(x,x_{0})\rightarrow\infty$, where $x_{0}$ is a fixed point on the
spherically symmetric hypersurface $\Sigma$. This means that
$\Sigma$ is asymptotically flat.

\section{Self-adjoint extension of the Laplacian on the quantum layers}
\renewcommand{\thesection}{\arabic{section}}
\renewcommand{\theequation}{\thesection.\arabic{equation}}
\setcounter{equation}{0} \setcounter{maintheorem}{0}

As in \cite{ppd,ll}, from the definition (\ref{2.1}), the metric
tensor of the layer as a submanifold of $R^{m}$ satisfies
 \begin{eqnarray} \label{3.1}
G_{ij}=\left\{
\begin{array}{lll}
(\delta^{\sigma}_{i}-uh^{\sigma}_{i})(\delta^{\rho}_{\sigma}-uh^{\rho}_{\sigma})g_{\rho{j}},
\quad \quad 1\leq{i,j}\leq{m-1},\\
0,  \qquad\qquad\qquad\qquad\qquad\qquad  i~{\rm{or}}~j=m,\\
1,  \qquad\qquad\qquad \qquad\qquad \qquad i=j=m, &\quad
\end{array}
\right.
\end{eqnarray}
which implies the metric matrix has the block form
$$ (G_{ij})=\left(
 \begin{array}{cc}
   G_{\mu\nu} & 0  \\
   0 & 1 \\
  \end{array}
  \right)  {\qquad \rm{with}}\quad
  G_{\mu\nu}=\left(\delta^{\sigma}_{\mu}-uh^{\sigma}_{\mu}\right)\left(\delta^{\rho}_{\sigma}-uh^{\rho}_{\sigma}\right)g_{\rho{\nu}}, \quad
  1\leq{\mu,\nu}\leq{m-1}.$$
 Then by (\ref{3.1}) we obtain
\begin{eqnarray} \label{3.2}
\det(G_{AB})=\left[\det(1-uA)\right]^{2}\det(g_{\mu\nu}).
\end{eqnarray}
Since the eigenvalues of the matrix of the Weingarten map are the
principle curvatures $k_{s}$, $k_{\theta}$, we have
\begin{eqnarray} \label{3.3}
\det(1-uA)=(1-uk_{s})(1-uk_{\theta})^{m-2},
\end{eqnarray}
where $k_{s}$ and $k_{\theta}$ are given by (\ref{2.5}). By
assumption A2, the entries $G_{\mu\nu}$ of the matrix $G$ can be
estimated by
\begin{eqnarray} \label{3.4}
C_{-}g_{\mu\nu}\leq{G_{\mu\nu}}\leq{C_{+}}g_{\mu\nu},
\end{eqnarray}
where $C_{\pm}:=(1\pm{a\rho_{m}^{-1}})^{2}$ with
$0<C_{-}\leq1\leq{C_{+}}<4$. So, assumption A2 makes sure that the
mapping $\Phi$ is nonsingular, which implies the mapping $\Phi$
induces a Riemannian metric $G$ on $\Omega$. Hence, we know that the
mapping $\Phi$ is a diffeomorphism under assumptions A1 and A2.

There is an interesting truth we would like to point out here. From
the last section, we know that the Riemannian metric of the
spherically symmetric hypersurface $\Sigma$ can be expressed as
$ds^{2}+r^{2}|d\xi|^{2}$ with $|d\xi|^{2}$ the round metric on
$\mathbb{S}^{m-2}$ under the parametrization (\ref{2.3}), then by
(\ref{3.1}) the Riemannian metric of the quantum layer $\Omega$
built over $\Sigma$ can be written as
$du^{2}+ds^{2}+r^{2}|d\xi|^{2}$, which implies $\Omega$ is also
cylindrically symmetric.

For convenience, let $x_{1}:=s$,
$x_{2}:=\theta_{1},\ldots,x_{m-1}:=\theta_{m-2},x_{m}:=u$,  then in
the coordinate system $\{x_{1},\ldots,x_{m}\}$ on $\Omega$, the
Laplacian $\Delta=\Delta_{\Omega}$ can be written as
\begin{eqnarray*}
\Delta=\frac{1}{\sqrt{\det(G_{ij})}}\sum_{\mu,\nu=1}^{m-1}\frac{\partial}{\partial{x}_{\mu}}\left(G^{\mu\nu}\sqrt{\det(G_{ij})}\frac{\partial}
{\partial{x}_{\nu}}\right)+\frac{1}{\sqrt{\det(G_{ij})}}\frac{\partial}{\partial{u}}\left(G^{mm}\sqrt{\det(G_{ij})}\frac{\partial}{\partial{u}}\right).
\end{eqnarray*}
Using (\ref{3.3}) we could split $\triangle$ into a sum of two
parts, $\triangle=\triangle_{1}+\triangle_{2}$, given by
\begin{eqnarray*}
\Delta_{1}:=\frac{1}{\sqrt{\det(G_{ij})}}\frac{\partial}{\partial{u}}\left(G^{mm}\sqrt{\det(G_{ij})}\frac{\partial}{\partial{u}}\right)=\frac{\partial^{2}}{\partial{u^{2}}}-
\left(\frac{k_{s}}{1-uk_{s}}+\frac{(m-2)k_{\theta}}{1-uk_{\theta}}\right)\frac{\partial}{\partial{u}}
\end{eqnarray*}
and
\begin{eqnarray*}
\Delta_{2}:=\Delta-\Delta_{1}=\frac{1}{\sqrt{\det(G_{ij})}}\sum_{\mu,\nu=1}^{m-1}\frac{\partial}{\partial{x}_{\mu}}\left(G^{\mu\nu}\sqrt{\det(G_{ij})}\frac{\partial}
{\partial{x}_{\nu}}\right).
\end{eqnarray*}
In the rest part of this section, we will show that this Laplacian
$\Delta=\Delta_{\Omega}$ can be extended to a self-adjoint operator
on the quantum layer $\Omega$, which is a noncompact noncomplete
Riemannian manifold. For any $E,F\in{C}^{\infty}_{0}(\Omega)$, the
set of all smooth functions with compact support on $\Omega$, we
define the $L^{2}$ inner product $(\cdot,\cdot)$ as follows
\begin{eqnarray*}
(F,G)=\int_{\Omega}FGd\Omega,
\end{eqnarray*}
where $d\Omega$ is the volume element of the quantum layer $\Omega$.
Correspondingly, the norm $\|E\|$ could be defined by
$\|E\|:=\sqrt{(E,E)}$. Moreover, if $E,F$ are differentiable, we
define
\begin{eqnarray*}
(\nabla{E},\nabla{F})=\int_{\Omega}\left(\sum_{\mu,\nu=1}^{m-1}G^{\mu\nu}\frac{\partial{E}}{\partial{x_{\mu}}}\frac{\partial{F}}{\partial{x_{\nu}}}
+\frac{\partial{E}}{\partial{u}}\frac{\partial{F}}{\partial{u}}\right)d\Omega.
\end{eqnarray*}
Also, we define $\|\nabla{E}\|=\sqrt{(\nabla{E},\nabla{E})}$. Then
as the proof of proposition 2.1 in \cite{ll}, for any
$E,F\in{W^{1,2}_{0}}(\Omega)$, the space which is the closure of the
space ${C}^{\infty}_{0}(\Omega)$ under the norm
\begin{eqnarray*}
\|E\|_{W^{1,2}_{0}(\Omega)}=\sqrt{\|E\|^{2}+\|\nabla{E}\|^{2}},
\end{eqnarray*}
the sesquilinear form $Q_{1}(E,F):=(\nabla{E},\nabla{F})$ is a
quadratic form of a unique self-adjoint operator. Such an operator
is an extension of $\Delta$, which we still denote as $\Delta$.
Hence, we can use (\ref{1.1}) and (\ref{1.2}) to compute
$\sigma_{0}$ and $\sigma_{ess}$ for the quantum layer $\Omega$,
respectively.

However, generally it is complicated to construct trial functions on
the quantum layer $\Omega$ directly, our strategy to solve this
difficulty is the following: by introducing the unitary
transformation $\psi\rightarrow\psi\Phi$ with $\Phi$ defined by
(\ref{2.1}), we may identify the Hilbert space $L^{2}(\Omega)$ with
$\mathcal {H}:=L^{2}(\Omega_{0},d\Omega)$ and the Laplacian
$\Delta=\Delta_{\Omega}$ with the self-adjoint operator $H$
associated with the quadratic form $Q_{2}$ on $\mathcal{H}$ defined
by
\begin{eqnarray*}
&&Q_{2}(\psi,\psi):=\int_{\Omega_{0}}\overline{\psi_{,i}}G^{ij}\psi_{,j}d\Omega,\\
&&\psi\in{DomQ_{2}}:=\left\{\psi\in{W}^{1,2}(\Omega_{0},d\Omega)|\psi(q,u)=0
~{\rm{for ~a.e.}}~ (q,u)\in\Sigma_{0}\times\{\pm{a}\}\right\},
\end{eqnarray*}
here $\psi(x)$ for $x\in\partial\Omega_{0}$ means the corresponding
trace of the function $\psi$ on the boundary.

\section{Proof of main theorem}
\renewcommand{\thesection}{\arabic{section}}
\renewcommand{\theequation}{\thesection.\arabic{equation}}
\setcounter{equation}{0} \setcounter{maintheorem}{0}

Under assumptions A2 and A3, as the proof of theorem 3.1 in
\cite{ll}, we can prove the following.

\begin{theorem} \label{theorem2}
 Assume $\Omega$ is a quantum layer built over an
oriented hypersurface immersed in $R^{m}$ $(3\leq{m}<\infty)$, then
under assumptions A2 and A3, we have
$\sigma_{ess}\geq(\frac{\pi}{2a})^{2}$.
\end{theorem}

 In order to prove our main
theorem later, we need the following lemma.

\begin{lemma}(\cite{ll}) \label{lemma1}
Let $a>0$ be a positive number and let $k_{1}=\frac{\pi}{2a}$. Let
$\chi_{1}(u)=\cos(k_{1}u)$, let
\begin{eqnarray*}
\eta_{k}=\int_{-a}^{a}u^{k}(\chi_{1,u}^{2}-k_{1}^{2}\chi_{1}^{2})du,
\qquad \forall k\geq0,
\end{eqnarray*}
where $\chi_{1,u}$ denotes the derivative of $\chi_{1}$ with respect
to $u$. Then
\begin{eqnarray*}
\eta_{k}=\left\{
\begin{array}{ll}
0,
\qquad \qquad \qquad \qquad \qquad if ~k~is~odd,~or~k=0;\\
\frac{1}{2}\frac{(k)!}{(2k_{1})^{k-1}}\sum_{l=1}^{k/2}\frac{(-1)^{k/2-l}\pi^{2l-1}}{(2l-1)!},
\qquad if ~k\neq0~is~even.  &\quad
\end{array}
\right.
\end{eqnarray*}
Furthermore, $\eta_{k}>0$ if $k\neq0$ is even.
\end{lemma}

For the spherically symmetric hypersurface $\Sigma\subseteq{R^{m}}$
($3\leq{m}>\infty$) with a pole, we define a quantity
$\mathcal{K}_{m-2}$ by
\begin{eqnarray} \label{4.1}
\mathcal{K}_{m-2}:=\sum_{k=1}^{[(m-1)/2]}\eta_{2k}c_{2k}(A), \qquad
3\leq{m}<\infty,
\end{eqnarray}
where $\eta_{k}$ for $k\geq1$ is given in Lemma \ref{lemma2},
$[(m-1)/2]$ is the integer part of $(m-1)/2$, and $c_{k}(A)$ is the
$k$th elementary symmetric polynomial of the second fundamental form
$A$ of $\Sigma$. When $m=4$, we can obtain the following lemma.

\begin{lemma} \label{lemma2}
If $\mathcal{K}_{2}$ defined by (\ref{4.1}) is integrable on
 a $3$-dimensional spherically symmetric hypersurface $\Sigma$ with a pole
embedded in $R^4$, and $\Sigma$ is not a hyperplane, then we have
 $\\$ (1) $\Sigma$ is non-parabolic,
 $\\$ (2) $\lim\limits_{s\rightarrow\infty}\frac{r(s)}{s}=1,$
 $\\$ (3) $\int_{0}^{\infty}k_{s}(s)k_{\theta}(s)r(s)ds=0$, which
 implies there exists at least one domain on $\Sigma$ such that
 $k_{s}$ and $k_{\theta}$ have the same sign on this domain,
$\\$ here $r(s)$ is given by (\ref{2.3}) satisfying (\ref{2.4}), and
$k_{s}$, $k_{\theta}$ are given by (\ref{2.5}).
\end{lemma}
\begin{proof} Since $\mathcal{K}_{2}$ is
integrable on $\Sigma$ which can be parametrized by (\ref{2.3}) with
the requirement (\ref{2.4}), then we know that
$\int_{\Sigma_{0}}k_{s}(s)k_{\theta}(s)d\Sigma$ and
$\int_{\Sigma_{0}}k_{\theta}^{2}(s)d\Sigma$ are finite, which
implies $\int_{0}^{\infty}k_{s}(s)k_{\theta}(s)r^{2}(s)ds$ and
$\int_{0}^{\infty}k_{\theta}^{2}(s)r^{2}(s)ds$ are finite. By
(\ref{2.6}), we could obtain
\begin{eqnarray*}
r'(s)r(s)=\int_{0}^{s}\left(r'(v)\right)^{2}dv-\int_{0}^{s}k_{s}(v)k_{\theta}(v)r^{2}(v)dv,
\end{eqnarray*}
together with (\ref{2.4}) and (\ref{2.5}), it follows that
\begin{eqnarray} \label{4.2}
r'(s)r(s)=s-\int_{0}^{s}k_{\theta}^{2}(v)r^{2}(v)dv-\int_{0}^{s}k_{s}(v)k_{\theta}(v)r^{2}(v)dv.
\end{eqnarray}
Let $D$ be
\begin{eqnarray*}
D:=\int_{0}^{\infty}k_{s}(s)k_{\theta}(s)r^{2}(s)ds+\int_{0}^{\infty}k_{\theta}^{2}(s)r^{2}(s)ds,
\end{eqnarray*}
then there exists a constant $s_{0}>1$ such that for any
$s\geq{s_{0}}$, we have
\begin{eqnarray*}
\left|\int_{0}^{s}k_{s}(v)k_{\theta}(v)r^{2}(v)dv+\int_{0}^{s}k_{\theta}^{2}(v)r^{2}(v)dv-D\right|\leq\frac{1}{100}.
\end{eqnarray*}
Integrating (\ref{4.2}) from $s_{0}$ to $s$ results in
\begin{eqnarray} \label{4.3}
s^{2}-s_{0}^{2}-\left(2D+\frac{1}{50}\right)(s-s_{0})+r^{2}(s_{0})\leq{r^{2}(s)}\leq{s^{2}}-s_{0}^{2}+\left(2|D|+\frac{1}{50}\right)(s-s_{0})+r^{2}(s_{0}),
\end{eqnarray}
for any $s\geq{s_{0}}$.

On the other hand, from (\ref{4.2}), we also have
\begin{eqnarray*}
\lim\limits_{s\rightarrow\infty}\frac{r'(s)r(s)}{s}=1-\lim\limits_{s\rightarrow\infty}s^{-1}
\left[\int_{0}^{s}k_{\theta}^{2}(v)r^{2}(v)dv+\int_{0}^{s}k_{s}(v)k_{\theta}(v)r^{2}(v)dv\right]=1,
\end{eqnarray*}
together with (\ref{4.3}), it follows that
\begin{eqnarray} \label{4.4}
r'(\infty):=\lim\limits_{r\rightarrow\infty}r'(s)=1.
\end{eqnarray}
 By (\ref{2.4}), (\ref{2.6}) and (\ref{4.4}), we have
\begin{eqnarray}  \label{4.5}
\int_{0}^{\infty}k_{s}(s)k_{\theta}(s)r(s)ds=0 \qquad {\rm{and}}
\qquad \lim\limits_{s\rightarrow\infty}z'(s)=0.
\end{eqnarray}
Now, we would like to prove the first assertion by using the
estimate (\ref{4.3}), however, before that some useful facts about
parabolicity should be given first.
\begin{defn} A complete manifold is said to be non-parabolic if it
admits a non-constant positive superharmonic function. Otherwise it
is said to be parabolic.
\end{defn}
\begin{lemma} (\cite{grig2,grig3,td}) Let Riemannian manifold $M$ be
geodesically complete, and for some $x\in{M}$,
\begin{eqnarray} \label{4.6}
\int_{1}^{\infty}\frac{1}{S(x,\rho)}d\rho=\infty
\end{eqnarray}
with $S(x,\rho)$ the boundary area of the geodesic sphere
$\partial{B(x,\rho)}$. Then $M$ is parabolic.
\end{lemma}
In general, (\ref{4.6}) is not necessary for parabolicity, however,
in \cite{grig1}, Grigor'yan has shown that for a spherically
symmetric manifold $\widetilde{M}$ with  a pole, (\ref{4.6}) is also
a necessary condition for $\widetilde{M}$ being parabolic. Hence, if
we want to show $\Sigma$ is non-parabolic here, it suffices to prove
there exists some $x\in{\Sigma}$ such that
\begin{eqnarray*}
\int_{1}^{\infty}\frac{1}{S(x,t)}dt<\infty
\end{eqnarray*}
with $S(x,t)$ the area of the boundary of the geodesic ball $B(x,t)$
centered at $x$ with radius $t$. Now, for the $3$-dimensional
spherically symmetric hypersurface $\Sigma$ with a pole $o$, choose
$x$ to be the pole $o$, then the area $S(o,t)$ can be expressed by
$S(o,t)=w_{2}r^{2}(t)$ with $w_{2}$ the $2$-volume of the unit
sphere in $R^{3}$. So, by applying (\ref{4.3}), we have
\begin{eqnarray*}
\int_{1}^{\infty}\frac{1}{S(o,t)}dt\leq\int_{1}^{s_{1}}\frac{1}{w_{2}r^{2}(t)}dt+\frac{1}{w_{2}}
\int_{s_{1}}^{\infty}\frac{1}{s^{2}-s_{0}^{2}-\left(2D+\frac{1}{50}\right)(s-s_{0})+r^{2}(s_{0})}ds<\infty,
\end{eqnarray*}
where $s_{1}$ is chosen to be
\begin{eqnarray*}
s_{1}:=\left\{
\begin{array}{lll}
s_{0}, \qquad\quad if ~~\aleph\leq0, \\
\\
\max\left\{s_{0},\frac{1}{100}+D+\sqrt{(D+\frac{1}{100})^{2}-(2D+\frac{1}{50})s_{0}+s_{0}^{2}-r^{2}(s_{0})}\right\},
\qquad if ~~\aleph>0, &\quad
\end{array}
\right.
\end{eqnarray*}
with
$\aleph:=-r^{2}(s_{0})-(2D+\frac{1}{50})s_{0}+s_{0}^{2}+(D+\frac{1}{100})^{2}$.
Our proof is finished.
\end{proof}

By using Lemma \ref{lemma2}, we could obtain a result on the growth
speed of the volume of a geodesic ball of a 3-dimensional
spherically symmetric hypersurface related to the integrability of
$\mathcal{K}_{2}$ as follows.

\begin{corollary} \label{corollary1}
 Let $\Sigma$ be a $3$-dimensional spherically symmetric
manifold with a pole $o$ embedded in $R^{4}$, if in addition
$\mathcal{K}_{2}$ defined by (\ref{4.1}) is integrable on $\Sigma$,
then the volume $V(o,s)$ of the geodesic ball $B(o,s)$ with center o
and radius $s$ has cubic growth as $s$ large enough.
 \end{corollary}
 \begin{proof}
 We can set up the global geodesic polar coordinate chart centered at
$o$ for $\Sigma$ as before, consequently, the volume of the geodesic
ball $B(o,s)$ is given by
\begin{eqnarray*}
V(o,s)=\int_{0}^{s}\int_{\mathbb{S}^{2}}r^{2}(v)d{\mathbb{S}^{2}}dv,
\end{eqnarray*}
where $r$ satisfies (\ref{2.6}). By applying (\ref{4.3}), we have
\begin{eqnarray*}
\frac{w_{2}s^{3}}{3}-\left(D+\frac{1}{100}\right)s^{2}+c_{1}s\leq{V(o,s)}\leq\frac{w_{2}s^{3}}{3}+\left(|D|+\frac{1}{100}\right)s^{2}+c_{2}s,
\end{eqnarray*}
for any $s\geq{s_{0}}$, where
$c_{1}:=\left[r_{0}^{2}-s_{0}^{2}+\left(2D+1/50\right)s_{0}\right]w_{2}$
and
$c_{2}:=\left[r_{0}^{2}-s_{0}^{2}-\left(2|D|+1/50\right)s_{0}\right]w_{2}$.
This implies $V(o,s)$ has the cubic growth as $s$ large enough.
\end{proof}

Let $(M,g)$ be an $n$-dimensional complete Riemannian manifold,
denote by $B(q,r)$ the open geodesic ball centered at a
 point $q\in{M}$ with radius $r$ and by ${\rm{vol}}(B(q,r))$ its volume. Define
\begin{eqnarray*}
\alpha_{M}:=\lim\limits_{r\rightarrow\infty}\frac{{\rm{vol}}(B(q,r))}{v_{n}(1)r^{n}},
\end{eqnarray*}
with $v_{n}(1)$ the volume of the unit ball in $R^{n}$. It is not
difficult to prove $\alpha_{M}$ is independent of the choice of $q$,
which implies $\alpha_{M}$ is a global geometric invariant. We say
that $(M,g)$ has large volume growth provided $\alpha_{M}>0$. For
the spherically symmetric hypersurface $\Sigma$ with a pole $o$
embedded in $R^{4}$ with $\mathcal{K}_{2}$ integrable, by Corollary
\ref{corollary1} we have
\begin{eqnarray*}
\alpha_{\Sigma}=\lim\limits_{r\rightarrow\infty}\frac{V(o,r)}{v_{3}(1)r^{3}}=1>0,
\end{eqnarray*}
which implies $\Sigma$ has large volume growth provided
$\mathcal{K}_{2}$ is integrable.

Large volume growth assumption is common in deriving a prescribed
manifold with nonnegative Ricci curvature to be of finite
topological type. However, recently the author proved that a
complete open manifold with nonnegative Ricci curvature is of finite
topological type without the large volume growth assumption in
\cite{m}.

By using Lemmas \ref{lemma1} and \ref{lemma2}, we can prove the
following conclusion.
\begin{theorem} \label{theorem3}
Assume $\Omega$ is the quantum layer built over a spherically
symmetric hypersurface $\Sigma$ with a pole embedded in $R^{4}$, and
$\Sigma$ is not a hyperplane, if in addition $\mathcal{K}_{2}$ is
integrable on $\Sigma$ and
\begin{eqnarray*}
\int_{\Sigma}\mathcal{K}_{2}d\Sigma\leq0
\end{eqnarray*}
 with $\mathcal{K}_{2}$ defined by
(\ref{4.1}), then under assumptions A1 and A2, we have
$\sigma_{0}<(\frac{\pi}{2a})^{2}$.
\end{theorem}

\begin{proof} Here we use a similar method as that of theorem 5.1 in \cite{ppd}. Set
$\chi(u):=\sqrt{\frac{1}{a}}\cos(\frac{\pi{u}}{2a})=\sqrt{\frac{1}{a}}\chi_{1}(u)$.
We divide the proof into two steps:

(1) If $\int_{\Sigma}\mathcal{K}_{2}d\Sigma<0$, construct a trial
function $\Psi(s,u):=\varphi_{\sigma}(s)\chi(u)$, where
$\sigma\in(0,1]$ and
\begin{eqnarray} \label{4.7}
\varphi_{\sigma}(s):=\left\{
\begin{array}{lll}
1, \qquad\qquad\qquad\quad\quad ~if ~0<s\leq{s_{0}}, \\
\\
\min\left\{1,\frac{K_{0}(\sigma{s})}{K_{0}(\sigma{s_{0}})}\right\},
\qquad if ~s>{s_{0}}, &\quad
\end{array}
\right.
\end{eqnarray}
with $K_{0}(s)$ the Macdonald function (see \cite{mi}, Sec. 9.6).
Obviously, $\Psi(s,u)$ is continuous on $\Omega_{0}$, which implies
$\Psi\in{DomQ_{2}}$. By (\ref{1.1}) and the strategy explained at
the end of the last section, if we want prove
$\sigma_{0}<(\frac{\pi}{2a})^{2}$, it suffices to show that
\begin{eqnarray*}
-\int_{\Omega_{0}}\Psi(s,u)\Delta\Psi(s,u)d\Omega-\left(\frac{\pi}{2a}\right)^{2}
\int_{\Omega_{0}}\Psi^{2}(s,u)d\Omega
\end{eqnarray*}
is strictly negative.

By applying (\ref{3.2}), (\ref{3.3}) and Lemma \ref{lemma1}, we know
that
\begin{eqnarray*}
-\int_{\Omega_{0}}\Psi(s,u)\Delta_{2}\Psi(s,u)d\Omega-\left(\frac{\pi}{2a}\right)^{2}
\int_{\Omega_{0}}\Psi^{2}(s,u)d\Omega=\int_{\Sigma_{0}}(2k_{s}k_{\theta}+k_{\theta}^{2})(\varphi_{\sigma}(s))^{2}d\Sigma.
\end{eqnarray*}
Since $\mathcal{K}_{2}$ is integrable on $\Sigma$,
$|\varphi_{\sigma}(s)|\leq1$, and $\varphi_{\sigma}\rightarrow1$
pointwise as $\sigma\rightarrow0+$, then by the dominated
convergence theorem, we know that
\begin{eqnarray} \label{4.8}
-\int_{\Omega_{0}}\Psi(s,u)\Delta_{2}\Psi(s,u)d\Omega-\left(\frac{\pi}{2a}\right)^{2}
\int_{\Omega_{0}}\Psi^{2}(s,u)d\Omega\rightarrow\int_{\Sigma_{0}}(2k_{s}k_{\theta}+k_{\theta}^{2})d\Sigma=\int_{\Sigma}
\mathcal{K}_{2}d\Sigma
\end{eqnarray}
as $\sigma\rightarrow0+$.

On the other hand, an integration of (\ref{4.2}) together with the
fact that $\mathcal{K}_{2}$ is integrable on $\Sigma$ yields that
for any $s>0$, there exists a constant $c_{2}$ depending on the
value of $\int_{\Sigma}\mathcal{K}_{2}d\Sigma$ such that
\begin{eqnarray} \label{4.9}
r^{2}(s)\leq{s^2}+c_{3}s.
\end{eqnarray}
So, by (\ref{3.2}), (\ref{3.3}), (\ref{3.4}) and (\ref{4.9}), we
have
\begin{eqnarray} \label{4.10}
-\int_{\Omega_{0}}\Psi(s,u)\Delta_{1}\Psi(s,u)d\Omega&=&-\int_{-a}^{a}\int_{\Sigma_{0}}\left(\varphi'_{\sigma}(s)\chi(u)\right)^{2}
\frac{(1-uk_{\theta}(s))^{2}}{1-uk_{s}(s)}d\Sigma{du}\nonumber\\
&\leq&\frac{w_{2}C_{+}}{\sqrt{C_{-}}}\int_{0}^{\infty}\left(\varphi'_{\sigma}(s)\right)^{2}
(s^{2}+c_{3}s)ds.
\end{eqnarray}
However, by using Mathematica and properties of Macdonald function
given by
\begin{eqnarray*}
-2K'_{v}(z)=K_{v-1}(z)+K_{v+1}(z),\\
-\frac{2v}{z}K_{v}(z)=K_{v-1}(z)-K_{v+1}(z),\\
K_{0}(z)=-\log{z}+O(1), \qquad {\rm{as}}~z\rightarrow0,\\
K_{1}(z)=\frac{1}{z}+O(\log{z}), \qquad {\rm{as}}~z\rightarrow0,\\
\end{eqnarray*}
it follows that as $\sigma\rightarrow0+$, there exists a constant
$c_{4}$ such that
\begin{eqnarray*}
\int_{0}^{\infty}\left(\varphi'_{\sigma}(s)\right)^{2}s^{2}ds=\frac{1}{(K_{0}(\sigma{s_{0}}))^{2}}\int^{\infty}_{\sigma{s_{0}}}\left(K'_{0}(t)\right)^{2}
t^{2}dt\rightarrow\frac{3\pi^{2}}{32(K_{0}(\sigma{s_{0}}))^{2}}\rightarrow0
\end{eqnarray*}
and
\begin{eqnarray*}
\int_{0}^{\infty}\left(\varphi'_{\sigma}(s)\right)^{2}sds\leq\frac{c_{4}}{|\log\sigma{s_{0}}|}\rightarrow0.
\end{eqnarray*}
Substituting the above estimates in (\ref{4.10}) results in
\begin{eqnarray} \label{4.11}
-\int_{\Omega_{0}}\Psi(s,u)\triangle_{1}\Psi(s,u)d\Omega\rightarrow0
\end{eqnarray}
as $\sigma\rightarrow0+$. So, from (\ref{4.8}) and (\ref{4.11}), we
have
\begin{eqnarray*}
-\int_{\Omega_{0}}\Psi(s,u)\Delta\Psi(s,u)d\Omega-\left(\frac{\pi}{2a}\right)^{2}
\int_{\Omega_{0}}\Psi^{2}(s,u)d\Omega\rightarrow\int_{\Sigma}
\mathcal{K}_{2}d\Sigma<0
\end{eqnarray*}
as $\sigma\rightarrow0+$, which implies
$\sigma_{0}<(\frac{\pi}{2a})^{2}$.

(2) If $\int_{\Sigma}\mathcal{K}_{2}d\Sigma=0$, construct a trial
function
$\Psi_{\sigma,\epsilon}:=(\varphi_{\sigma}(s)+\epsilon{j(q)u})\chi(u)$
with $\varphi_{\sigma}(s)$ defined by (\ref{4.7}) and
$j\in{C}_{0}^{\infty}\left((0,s_{0})\times\mathbb{S}^{2}\right)$.
Obviously, $\Psi_{\sigma,\epsilon}\in{DomQ_{2}}$. For convenience,
for any function $f\in{DomQ_{2}}$, let
\begin{eqnarray*}
Q_{3}[f]:=-\int_{\Omega_{0}}f\Delta{f}d\Omega-\left(\frac{\pi}{2a}\right)^{2}
\int_{\Omega_{0}}f^{2}d\Omega.
\end{eqnarray*}
By applying Lemma \ref{lemma1}, we have
\begin{eqnarray} \label{4.12}
Q_{3}[\Psi_{\sigma,\epsilon}]=Q_{3}[\varphi_{\sigma}(s)\chi(u)]-2\epsilon\int_{\Omega_{0}}j(k_{s}+2k_{\theta})d\Omega+\epsilon^{2}
Q_{3}[j(q)u\chi(u)].
\end{eqnarray}
The second term on the right hand side of (\ref{4.12}) can be made
nonzero by choosing $j$ supported on a compact subset of
$\Sigma_{0}$ where $(k_{s}+2k_{\theta})$ does not change sign. The
existence of this compact subset could be assured by Lemma
\ref{lemma2} (3) and the fact that we could choose $s_{0}$
arbitrarily large. So, if we choose the sign of $\epsilon$ in such a
way that the second term on the right hand side of (\ref{4.12}) is
negative, then, for sufficiently small $\epsilon$, the sum of the
last two terms of the right hand side of (\ref{4.12}) will be
negative. On the other hand, by the argument in (1), we know that
\begin{eqnarray*}
Q_{3}[\varphi_{\sigma}(s)\chi(u)]\rightarrow\int_{\Sigma}\mathcal{K}_{2}d\Sigma
\end{eqnarray*}
as $\sigma\rightarrow0+$. Hence, we have
$Q_{3}[\Psi_{\sigma,\epsilon}]<0$ as $\sigma\rightarrow0+$ and
$\epsilon$ sufficiently small, which implies
$\sigma_{0}<(\frac{\pi}{2a})^{2}$.

Our proof is finished.
\end{proof}

So, by Theorems \ref{theorem2} and \ref{theorem3}, we have
\begin{corollary}
Theorem \ref{theorem1} is true.
\end{corollary}

\begin{remark} \label{remark} \rm{ The existence of the ground state of quantum layers built over
 submanifolds of high dimensional Euclidean space has
been obtained in \cite{ll,ll2} under some assumptions therein, but
the parabolicity of the reference submanifold is necessary in those
assumptions, however, here our $3$-dimensional reference
hypersurface $\Sigma$ of $R^4$ is non-parabolic by Lemma
\ref{lemma2}. So, the existence of the ground state of the
cylindrically symmetric quantum layers considered here can not be
obtained by the results in \cite{ll,ll2}, which indicates that
Theorem \ref{theorem1} can be seen as a complement to those existent
results for higher dimensional quantum layers or quantum tubes. }
\end{remark}

\end{document}